\documentclass[envcountsame,runningheads,notitlepage]{llncs}

\usepackage[a4paper, margin=1.1in]{geometry}
\usepackage{hyperref}
\usepackage{amsfonts,amsmath,amscd,amssymb}
\title{The Scholz conjecture on addition chains is true for $n=2^m(23)+7$, $m \in \mathbb{N}^*$}
\titlerunning{The Scholz conjecture for $n=2^m(23)+7$, $m \in \mathbb{N}^*$}
\date{}

\author{
  Amadou TALL\inst{1}
}

\institute{Departement de Mathématiques et Informatique\\Université Cheikh Anta Diop de Dakar\\
  \href{mailto:amadou7.tall@ucad.edu.sn}{amadou7.tall@ucad.edu.sn}
}  

\begin{document}
  \maketitle

\begin{abstract}
The Scholz conjecture on addition chains states that $\ell(2^n-1) \leq \ell(n) + n -1$ for all integers $n$ where $\ell(n)$ stands for the minimal length of all addition chains for $n$. It is proven to hold for infinite sets of integers.  In this paper, we will prove that the conjecture still holds for $n=2^m(23)+7$. It is the first set of integers given by Thurber \cite{9} to prove that there are an infinity of integers satisfying  $\ell(2n) = \ell(n)$. Later on, Thurber \cite{4} give a second set of integers with the same properties ($n=2^{2m+k+7} + 2^{2m+k+5} + 2^{m+k+4} + 2^{m+k+3} + 2^{m+2} + 2^{m+1} + 1$). We will prove that the conjecture holds for them as well.
\end{abstract}


\section*{Introduction}
Exponentiation is a key operation in mathematics. It is often seen as a number of multiplications. Let $x$ be an integer, finding $x^n$ for a given $n$ seems to be easy. We need to do $n$ multiplications. But what if $n$ is getting very large. 
It is a very important operation in many areas like cryptography. The best tool know to do fast exponentiation is addition chain. 

\begin{definition}
An addition chain for an integer $n$ is a set $\mathcal{C} =\{a_0=1,a_1,a_2,\ldots,a_r\}$ such that the last element $a_r$ is $n$ and every element $a_k$ is the sum of two previous elements $a_i$ and $a_j$. The integer $r$ is the length of the chain and is denoted $\ell(c)$.
\end{definition}
One can easily see that there can be many addition chains of different length for the same integer $n$ and finding the minimal length for all addition chains of $n$ (denoted $\ell(n)$) is a NP-complete problem. Just to give an example, on can think that the fastest way to reach $2n$ is by reaching $n$ and then doubling it, which means that 
$\ell(2n) = \ell(n) +1$. It has been proven that there are infinitely many integers which don't follow that rule. We have $\ell(382) =  \ell(191) =11$.  The most famous conjecture on addition chain is giving an upper bound on the length of addition chains for integers with only ''1''s in their binary expansion ($n = 2^a-1$ for some $a$).  
 
 \begin{conjecture}
 The Scholz conjecture on addition chains states that
 $$
 \ell(2^n-1) \leq \ell(n) + n - 1, \ \forall n \in \mathbb{N}
 $$
 \end{conjecture}
 
 Several decades of research and the conjecture has been proved to hold for $n < 5784689$. It has been proven to hold for several sets of integers.  We know that if it is true for $n$ and that $\ell(2n) = \ell(n) + 1$, then it is also true for $2n$. We would like to know if it holds when $\ell(2n) = \ell(n)$. An answer for the general case doesn't exist yet.

In this paper, we will investigate the two cases.
\begin{enumerate}
\item We will give a very simple proof  that the conjecture holds for $2n$ if it is for $n$ and $\ell(2n) = \ell(n) + 1$.
\item We will also prove that it holds  for integers $n = 2^m(23) + 7$ which are known to be among the first infinite sets of integers verifying $\ell(2n) = \ell(n)$.
\end{enumerate} 

The main result from Thurber \cite{9} is
\begin{theorem}
Let $m$ be an integer greater than $5$. Then 
$$
\ell(2^m(23) + 7) = m +8 = \ell(2 \times (2^m(23) + 7) ) 
$$
\end{theorem}

Let us remind the factor method, which is a method to construct a chain for the product $nm$ based on chains for $n$ and $m$. 
  \begin{definition}
 Let $c_1$ and $c_2$ be  addition chains respectively for $n_1$ and $n_2$.
 Then $c_1 \times c_2$ is an addition chain for $n_1 \times n_2$ of length $\ell(c_1) + \ell(c_2)$
 where $\times$ is \emph{defined} as follows:
 
 if $ \, c_1=\{a_0,~a_1,~\ldots,~a_r\} \, $ and $ \, c_2=\{b_0,~b_1,~\ldots,~b_l\} \, $, then
 $$
 c_1 \times c_2=\{a_0,~a_1,~\ldots,~a_r,~a_r \times b_1,~a_r \times b_2,~\times,~a_r \times b_l\}.
 $$
 \end{definition}
 The length of the new chain is the sum of the length of the chains, meaning that $\ell(mn) \leq \ell(n) + \ell(m)$.

\section{Our contribution}

\begin{theorem}
Let $n$ be a positive integer satisfying $\ell(2n) = \ell(n)+1$. If the Scholz conjecture holds for $n$, then it also holds for $2n$.
\end{theorem}
\begin{proof}
Let $n$ be a positive integer which satisfies $\ell(2n) = \ell(n)+1$ and $ \ell(2^n-1) \leq \ell(n) + n - 1$. Let $n_0 = 2n$ be another positive integer, we have 
$$
2^{n_0} - 1 = (2^n-1)(2^n+1), 
$$
using the factor method, we can deduce a chain for $2^{n_0} - 1$ of length 
$$
\ell(n) + n - 1 + n+1 = \ell(n) + 2n = \ell(n_0) + n_0 -1.
$$
\end{proof}

\begin{theorem}
The Scholz conjecture holds for all integers of the form $U_m = 2^m(23)+7$.
\end{theorem}
It will be proven by induction on $m$
\begin{proof}
\begin{enumerate}
\item Thanks to the computational results, the Scholz conjecture holds for $U_m$, $m \in \{1,2,3,4,5\}$
\item A chain for $X_m +3$ (with $X_m = 2^m(23)$) of length $m+7$ can be obtain as follows
$$
\mathcal{C_{X_m}}=\{1,2,3,5,10,20,23,2 \times 23, \cdots, 2^m \times 23, 2^m \times 23 +3\}
$$
\item We know that 
$$
U_{m+1} = 2^{m+1}(23) + 7 
$$
so,
$$
X_m = U_m -7 = \frac{U_{m+1}-1}{2} -3 
$$
which leads to 
\begin{align*}
2^{U_{m+1}} - 1 &= 2^{2(X_m+3)+1}-1 \\
 &= 2(2^{2(X_m+3)}-1)+1 \\
 &= 2(2^{X_m+3}-1)(2^{X_m+3}+1)+1
\end{align*}

From the chain of $(2^{X_m+3}-1)$ and using the factor method, we can then deduce a chain for $2^{U_{m+1}} - 1$ of length 

\begin{align*}
X_m+3 + 1 +1 +1 + \ell(2^{X_m+3}-1) &= X_m+6 + \ell(X_m+3) + X_m+3 - 1 \\
 &= 2X_m+8 + m  \\
 &= (2x_m+7) + (m+9) - 1 \\
 & = U_{m+1} + \ell(U_{m+1}) - 1
\end{align*}

\end{enumerate}
\end{proof}

Two years later, Thurber has also stated that 
\begin{theorem}
For each $m\geq 1$, the set of integers with $v(n) =7$ and $n$ of the binary form $n = 101\cdots m \cdots 11 \cdots k \cdots 11 \cdots m \cdots 1$ (where $k \geq 3$) is an infinite class of integers which $\ell(2n) = \ell(n)$.
\end{theorem}
He prove that $\ell(n) = \lambda(n) + 4$. We will also prove that the Scholz conjecture holds for such $n$.

\begin{theorem}
The Scholz conjecture on addition chain holds for $n=2^{2m+k+7} + 2^{2m+k+5} + 2^{m+k+4} + 2^{m+k+3} + 2^{m+2} + 2^{m+1} + 1$.
\end{theorem}
\begin{proof}
Let us rewrite 
\begin{align*}
n &=2^{2m+k+7} + 2^{2m+k+5} + 2^{m+k+4} + 2^{m+k+3} + 2^{m+2} + 2^{m+1} + 1 \\
&= 2^{m+k+3}(2^{m+4}+2+2^{m+2}+1) + (2^{m+2}+1 + 2^{m+1}) \\
&= 2^{m+k+3}(2(2^{m+3}+1)+(2^{m+2}+1)) + (2^{m+2}+1) + 2^{m+1}
\end{align*}

A minimal chain for $n$  is 
$$
\mathcal{c} = \{1,2,\ldots,2^{m+1},2^{m+2},2^{m+2}+1, 2^{m+3}+1,2^{m+1} + (2^{m+2}+1)=\beta,
$$
$$
 2(2^{m+3}+1),  2(2^{m+3}+1) +(2^{m+2}+1)=\alpha, 
 $$
$$ 
 2\alpha,2^2\alpha,\ldots, 2^{m+k+3}\alpha, n = 2^{m+k+3}\alpha + \beta\}.
$$
which is of length 
$$
m+2+1+1+1+1+1+m+k+3+1=2m+k+11 = \lambda(n) + 4.
$$
Using the following rules:
\begin{enumerate}
\item If $A=2B$, then we will get a chain for $2^A-1 = (2^B-1)(2^B+1)$ by adding $B+1$ steps to a chain for $2^B -1$
\item If $A = B + 1$, then we will add two steps ($2^A-1 = 2(2^B-1)+1$) 
\item $2^{m+3}+1 = 2^{m+2}+ 2^{m+2}+1$, so
$$
2^{2^{m+3}+1}-1 = 2^{2^{m+2}}(2^{2^{m+2}+1}-1)+(2^{2^{m+2}}+1)
$$
\item $2^{m+2} + 2^{m+1} + 1 = (2^{m+2}+1) + 2^{m+1}$, so 
$$
2^{2^{m+2}+2^{m+1}+1}-1 = 2^{2^{m+1}}(2^{2^{m+2}+1}-1)+(2^{2^{m+1}}-1) 
$$
\end{enumerate}
We can then construct a chain $\mathcal{C}$ for $2^n-1$ of length
	\begin{center}
\begin{tabular}{ |c|c|c|c| } 
 \hline
 value & step & additional steps & Comment \\ \hline
   $1$ & DBL &   &   \\ \hline
   $2$ & DBL &  &   \\ \hline
    $\ldots$ & DBL &   &  \\ \hline
      $2^{m+1}$ & DBL &   &   \\ \hline
   $2^{m+2}$ & DBL & $2^{m+2} + m + 2 -1$ &   \\ \hline
       $2^{m+2}+1$ & +1 & $2$ &   \\ \hline
    $2^{m+3}+1$ & $+2^{m+2}$ & $2^{m+2}+1$ &   \\ \hline
      $2^{m+2}+2^{m+1}+1 = \beta$ & $2^{m+1}$ & $1$ &  $2^{2^{m+1}}(2^{2^{m+2}+1}-1) \in \mathcal{C}$ \\ \hline
   $2(2^{m+3}+1)$ & DBL & $2^{m+3}+1+1$ &   \\ \hline
    $2(2^{m+3}+1)+(2^{m+2}+1) = \alpha$ & $+2^{m+2}+1$ & $2^{m+2}+1+1$ &    \\ \hline
      $2\alpha$ & DBL & $\alpha +1$ &   \\ \hline
   $2^2\alpha$ & DBL & $2\alpha +1$ &   \\ \hline
    $\ldots$ & DBL & $\ldots$ &  \\ \hline
      $2^{m+k+3} \alpha$ & DBL & $2^{m+k+3}\alpha +1$ &   \\ \hline
   $n = 2^{m+k+3} \alpha + \beta$ & $+\beta$ & $\beta +1$ &   \\ \hline
 
   TOTAL & $4$ small steps   & $ $  & $\ell(n)+n-1 = $ \\ \hline \hline

\end{tabular}
\end{center}

\begin{align*}
\ell &= 2^{m+2} + m+2-1 +2 + 2^{m+2}+1+2+2^{m+3}+ \\ &+1+1+2^{m+2}+1+1+(m+k+3)+\alpha(2^{m+k+3}-1) + \beta \\
&= 2m + k + 13 + 2^{m+4}+2^{m+2} + (2^{m+4}+2^{m+2}+2+1)(2^{m+k+3}-1)+ 2^{m+2}+2^{m+1}+1\\
&=  2m + k +14 + 2^{m+4} + 2^{m+3} + 2^{m+1} + 2^{2m+k+7} + 2^{2m+k+5} 2^{m+k+4} + 2^{m+k+3} - 2^{m+4} - 2^{m+2} -2-1 \\
&= (2m + k +11) + (2^{2m+k+7} + 2^{2m+k+5} 2^{m+k+4} + 2^{m+k+3} + 2^{m+2} + 2^{m+1}) = \ell(n) + n -1.
\end{align*}
\end{proof}

%
%
%
%
%
%
%

\section*{Conclusion}
We don't know if the Scholz conjecture holds for all integers with $\ell(2n) = \ell(n)$. It must be more difficult to 
investigate the case where $\ell(2n) < \ell(n)$.

\section*{Acknowledgment}
The author would like to thank IHES. The work was completed during his visiting fellowship.

\end{document}